\documentclass[11pt,a4paper]{amsart}

\usepackage[utf8]{inputenc}
\usepackage[english]{babel}

\usepackage{amsmath,amscd}
\usepackage{amssymb}
\usepackage{amsthm}

\usepackage{verbatim}

\usepackage{graphicx}

\usepackage{mathrsfs}
\usepackage[ocgcolorlinks, linkcolor=blue]{hyperref}

\usepackage{calc}
             {\begin{list}{\arabic{enumi}.}{\usecounter{enumi}%
              \setlength{\labelsep}{0.5em}%
              \settowidth{\labelwidth}{(\arabic{enumi})}%
              \setlength{\leftmargin}{\labelwidth+\labelsep}}}%
             {\end{list}}

\newtheorem{Theorem}{Theorem}[section]

\newtheorem{Lemma}[Theorem]{Lemma}

\newtheorem{Proposition}[Theorem]{Proposition}

\theoremstyle{definition}
\newtheorem{Definition}[Theorem]{Definition}

\newtheorem{Remark}[Theorem]{Remark}

\numberwithin{equation}{section}

\newcommand{\mR}{\mathbb{R}}                    
\newcommand{\mC}{\mathbb{C}}                    
\newcommand{\abs}[1]{\lvert #1 \rvert}          
\newcommand{\norm}[1]{\lVert #1 \rVert}         
\newcommand{\br}[1]{\langle #1 \rangle}         

\renewcommand{\Re}{\mathrm{Re}}
\renewcommand{\Im}{\mathrm{Im}}

\newcommand{\p}{\partial}

\newcommand{\eps}{\varepsilon}
\renewcommand{\phi}{\varphi}

\newcounter{sidenote}
\begin{document}

\title[Resolvent estimates for the magnetic Schr\"odinger operator]{Resolvent estimates for the magnetic Schr\"odinger operator in dimensions $\geq 2$} 

\author[C. J. Meroño]{Cristóbal J. Meroño}
\address{Universidad Politécnica de Madrid, ETSI Caminos, Departmento de Matemática e Informática, Campus Ciudad Universitaria, Calle del Prof. Aranguren, 3, 28040 Madrid}
\email{cj.merono@upm.es}

\author[L. Potenciano-Machado]{Leyter Potenciano-Machado}
\address{University of Jyvaskyla, Department of Mathematics and Statistics, PO Box 35, 40014 University of Jyvaskyla, Finland}
\email{leyter.m.potenciano@jyu.fi}

\author[M. Salo]{Mikko Salo}
\address{University of Jyvaskyla, Department of Mathematics and Statistics, PO Box 35, 40014 University of Jyvaskyla, Finland}
\email{mikko.j.salo@jyu.fi}




\begin{abstract}
It is well known that the resolvent of the free Schr{\"o}dinger operator on weighted $L^2$ spaces has norm decaying like $\lambda^{-\frac{1}{2}}$ at energy $\lambda$. There are several works proving analogous high frequency estimates for magnetic Schr\"odinger operators, with large long or short range potentials, in dimensions $n \geq 3$. We prove that the same estimates remain valid in all dimensions $n \geq 2$.
\end{abstract}

\maketitle

\section{Introduction}

Resolvent estimates for Schr\"odinger operators play a fundamental role in stationary scattering theory \cite{reedsimon, H2} and in inverse scattering \cite{Eskin2011}. They are also useful when proving Strichartz, smoothing, and dispersive estimates, eigenvalue estimates, as well as local energy decay for wave and Schr\"odinger equations (see e.g. \cite{EGS1, CardosoCuevasVodev1, RodnianskiTao,  BDSFK17,GHK17}). In many of these applications it is important to understand the high frequency behavior of the resolvent, i.e.\ how the norm bounds depend on the frequency (or energy).

A standard high frequency resolvent estimate for the free Schr\"odinger operator in $\mR^n$ (see e.g.  \cite{Agmon1975}, \cite[Section 7.1]{Yafaev}) states that
\begin{equation} \label{free_resolvent_estimates2}
\norm{\partial^{\alpha}((-\Delta-\lambda\pm i\eps)^{-1}f)}_{L^2_{-\delta}} \leq C \lambda^{\frac{\abs{\alpha}-1}{2}} \norm{f}_{L^2_{\delta}}
\end{equation}
where $\lambda \geq 1$, $0 < \eps \leq 1$, $\delta > 1/2$, $\abs{\alpha} \leq 1$ and  $C$ is independent of $\lambda$ and $f$. The spaces $L^2_{\delta} = L^2_{\delta}(\mR^n)$ are the weighted Agmon spaces, and their  norm is given by 
\begin{equation*}
\norm{f}_{L^2_{\delta}} = \norm{\br{x}^{\delta} f}_{L^2}
\end{equation*}
where $\br{x} = (1+\abs{x}^2)^{1/2}$ and $\delta \in \mR$.

In this work, we consider a first order perturbation of the Laplacian, the magnetic Schr\"odinger operator in $\mR^n$, $n \geq 2$, given by 
\begin{equation} \label{eq:hamiltonian}
H =  -\Delta + 2W \cdot D + (D \cdot W) + V 
\end{equation}
where $D = \frac{1}{i} \nabla $, the magnetic potential $W: \mR^n \to \mR^n$ is a vector field, and the electrostatic potential $V: \mR^n \to \mR$ is a function. We assume that $W \in L^\infty(\mR^n, \mR^n) $ and  $V \in L^\infty(\mR^n, \mR) $.

 A direct perturbation argument shows that \eqref{free_resolvent_estimates2} remains valid  when $-\Delta$ is replaced by the magnetic Schr\"odinger operator $H$, provided that for some $\sigma>0$,
\[
\norm{\br{x}^{1+\sigma} W}_{L^{\infty}} \text{ is sufficiently small}, \qquad \norm{\br{x}^{1+\sigma} V }_{L^{\infty}} < \infty,
\]
and provided that $\lambda$ is sufficiently large (depending on $V$).

If the magnetic potential $W$ is large, the perturbation argument fails, and several works have been devoted to understanding  high frequency resolvent estimates. The articles \cite{EGS1, EGS2, Goldberg} employ harmonic analysis methods and prove high frequency resolvent estimates assuming that $W$ is continuous and the potentials are of short range type. Analogous estimates were proved earlier in \cite[Th\'eor\`eme (5.1)]{Robert} for smooth long range potentials satisfying symbol type bounds, also when the Euclidean metric is replaced by an asymptotically Euclidean metric with no trapped geodesics. The proof was based on a microlocal version of the Mourre commutator method, which  in turn is an instance of a positive commutator method.

In the Euclidean case, the works \cite{CardosoCuevasVodev1, CardosoCuevasVodev2,Vodev} prove high frequency resolvent estimates for long and short range potentials having low regularity, with the most general results given in \cite{Vodev}. Their proofs involve positive commutator arguments combined with ODE techniques, including Carleman estimates, that are valid under low regularity assumptions.  We mention also the works \cite{TesisMiren,Miren14}, in which the Morawetz multiplier method, also related to  positive commutator arguments, is used to allow magnetic potentials with singularities.

Many of the previously mentioned works explicitly assume that the dimension is $n \geq 3$. It is the purpose of this article to show that high frequency resolvent estimates for the magnetic Schr\"odinger operator, under  low regularity assumptions like in \cite{Vodev}, remain valid in all dimensions $n \geq 2$.

 We now assume that the potentials $V$ and $W$  in \eqref{eq:hamiltonian} have both long range and short range parts satisfying the following conditions:
\begin{equation}  \label{shortlong}
 V= V^L+ V^S, \quad W = W^L +W^S, 
\end{equation}
\begin{equation} \label{longrange1} 
|\nabla V^L| \le  C \br{x}^{-1 - \sigma}, \quad | W^L| \le  C \br{x}^{- \sigma},   \quad  |\nabla  W^L| \le C \br{x}^{-1- \sigma}.  
\end{equation}
\begin{equation} \label{shortrange1}  
|V^S| \le C \br{x}^{-1- \sigma},  \quad |W^S| \le C \br{x}^{-1- \sigma},
\end{equation}
for some $\sigma > 0$. 
In some results we will use also the   stronger  condition
\begin{equation} \label{shortrange2} 
 |\nabla \cdot W^S| \le C \br{x}^{-1- \sigma}.
\end{equation}
 We can state now the main result in this work.
\begin{Theorem} \label{thm:main1}
Let  $n \geq 2$ and $z\in \mC$ with $\Re(z) = \lambda$,  $|\Im(z)| \le 1$ and $\Im(z) \neq 0$. Assume that that $W$ and $V$ satisfy \eqref{shortlong} -- \eqref{shortrange1}.  Then, for any $\delta > 1/2$, there exist positive constants $C = C(n, \delta)$ and $\lambda_0 =\lambda_0(n,\delta,V,W)$ such that for every $\lambda \geq \lambda_0$, the resolvent $R(z) = (H-z)^{-1}$ satisfies the estimate
\begin{equation} \label{eq:main_estimate1}
\norm{\partial^{\alpha_1} R(z) \partial^{\alpha_2} f}_{L^2_{-\delta}} \leq C \lambda^{\frac{\abs{\alpha_1} + \abs{\alpha_2}-1}{2}} \norm{ f}_{L^2_{\delta}},
\end{equation}
whenever $\abs{\alpha_1} , \abs{\alpha_2} \leq 1$ and $f \in L^2_{\delta}$. Moreover, if one also assumes the condition \eqref{shortrange2} on the short range magnetic potential, then the estimate
\begin{equation} \label{eq:main_estimate2}
\norm{\partial^{\alpha} R(z)f}_{L^2_{-\delta}} \leq C \lambda^{\frac{\abs{\alpha}-1}{2}} \norm{ f}_{L^2_{\delta}},
\end{equation}
 holds for every $|\alpha| \le 2$.
\end{Theorem}
Estimate \eqref{eq:main_estimate1}  is  analogous to the results in \cite[Theorem 1.1]{Vodev} but it holds also for $n = 2$. Condition $\eqref{shortrange2}$ already appears in \cite{CardosoCuevasVodev2}. However, the results in both the above mentioned  papers hold  under the  following slightly weaker conditions on the long range potentials:
 \begin{equation} \label{longrange2}  
   \abs{\partial_r V^L} \le C \br{x}^{-1- \sigma}, \quad | W^L| \le  C \br{x}^{- \sigma},  \quad|\partial_r  W^L| \le C \br{x}^{-1- \sigma}. 
\end{equation}
Here $\partial_r:=\frac{x}{|x|}\cdot \nabla$ denotes the radial derivative. 
In our case,  the main a priori estimates in this work (see Lemma   \ref{lemma_longrange_resolvent_carleman} and Proposition \ref{prop_shortrange_carleman} below) are also obtained under the weaker long range conditions \eqref{longrange2}. Essentially,  the stronger conditions \eqref{longrange1} are only needed  for the final density argument used to prove Theorem \ref{thm:main1}. 

We remark that is immediate to see that Theorem \ref{thm:main1} also holds for a Hamiltonian $H = (D + W)^2 +  V$, under the conditions \eqref{shortlong} --\eqref{shortrange1}, since the extra term $W^2$ can be considered as part of the electrostatic potential and  can be decomposed suitably in a short range and long range part.

In the proof of  Theorem \ref{thm:main1} the self-adjointness of $H$ is essential so that the resolvent $R(z)$ can be defined as a bounded operator in  $L^2$  for all $z\in \mathbb{C}$ with $\mathrm{Im}(z)\neq 0$. This does not impose further restrictions on the potentials, since $H$ is self-adjoint  for  $W \in L^\infty(\mR^n, \mR^n) $ and  $V \in L^\infty(\mR^n, \mR) $ (see Proposition \ref{prop_self_adjoint} for a short proof of this basic fact).  
By definition, for $\lambda>0$ in the spectrum of $H$, the operator $(H-\lambda)^{-1}$ cannot be defined as a bounded operator in $L^2$. Nonetheless, it is well known that the limiting absorption principle (see e.g.\ \cite{reedsimon, H2})  provides a way to define the resolvent operators 
\begin{equation}
R(\lambda \pm i0)f := \lim_{\Im(z) \to 0^{\pm}} R(z) f ,
\end{equation} 
as bounded operators from $L^2_{\delta}$ to $L^2_{-\delta}$ for $\delta > 1/2$.

Under certain restrictions, a limiting absorption principle is proved  in  \cite[Theorem 30.2.10]{H2} in the presence of  long range and short range magnetic potentials. Then, it follows from this result that the resolvent $R(\lambda \pm i0)$ will satisfy the same bounds as $R(z)$ in Theorem \ref{thm:main1}. We state this with more precision in the following theorem.
\begin{Theorem} \label{thm:main2}
Assume that the hypotheses from Theorem \ref{thm:main1} hold, together with  \eqref{shortrange2}. Additionally, assume that 
$W^S$  is continuous. Then there is a discrete set $\Lambda \subset \mR_+$ (which is empty if $W^L = V^L = 0$) such that the resolvent for $H$ at energy $\lambda \in \mR_+ \setminus \Lambda$ with $\lambda \geq \lambda_0$ satisfies 
\begin{equation} \label{eq:main_thm2}
\norm{\partial^{\alpha} R(\lambda \pm i0) f}_{L^2_{-\delta}} \leq C \lambda^{\frac{\abs{\alpha}-1}{2}} \norm{f}_{L^2_{\delta}},
\end{equation}
for any $\abs{\alpha} \leq 2$ and $f \in L^2_{\delta}$.
\end{Theorem}

Our proof of Theorem \ref{thm:main1} employs analogous methods to the ones used in \cite{RodnianskiTao,  CardosoCuevasVodev2, Vodev}. As in \cite{Vodev}, we begin by proving  a global Carleman type estimate for the case $W^S = 0$, $ V^S = 0$. To prove this estimate, we use  a positive commutator argument  based on the construction of a suitable conjugate operator as in \cite[Section 6.1]{RodnianskiTao}, and integration by parts.  The conjugate operator is chosen carefully in order to have an estimate that is valid in any dimension $n \geq 2$, and the argument is different from \cite{Vodev} (in fact we only become aware of the works \cite{CardosoCuevasVodev1, CardosoCuevasVodev2, Vodev} after the main part of this paper had been written).

The commutator argument is explained in Section  \ref{sec:commutator} and the proof of the Carleman estimate is given in Section  \ref{sec:longrange}.
This  estimate, stated in Lemma \ref{lemma_longrange_resolvent_carleman}, would already imply \eqref{eq:main_estimate2} for $|\alpha| \le 1$ under  stronger conditions on $W^S$ and $\nabla \cdot W^S$. Then, following \cite{Vodev}, in Section \ref{sec:shift} we shift the previous estimate to  lower index Sobolev spaces to prepare for including the low regularity  term $\nabla \cdot W^S$. To conclude the proof of Theorem  \ref{thm:main1}, in Section \ref{sec:shortrange} we  include the short range perturbation and we extend all the a priori estimates, which hold for  $C^\infty_c$ functions,  to appropriate spaces using the Friedrichs lemma.  Finally, Theorem \ref{thm:main2} is proved  in the last section.

\subsection*{Acknowledgements}

C.M.\ was supported by Spanish government predoctoral grant BES-2015-074055 and projects MTM2014-57769-C3-1-P and MTM2017-85934-C3-3-P. L.P.\ and M.S.\ were supported by the Academy of Finland (Centre of Excellence in Inverse Modelling and Imaging, grant numbers 284715 and 309963) and by the European Research Council under Horizon 2020 (ERC CoG 770924).

\section{The commutator method} \label{sec:commutator}

We will first describe a positive commutator argument for the free resolvent following the presentation in \cite[Section 6]{RodnianskiTao}. Define 
\[
P := -\Delta.
\]
We will construct a first order differential operator $A$ (``conjugate operator") such that $i[P,A]$ is positive. If this is true, then for any $\lambda \in \mR$ and for any test function $u \in C^{\infty}_c(\mR^n)$ we have 
\[
(i[P,A]u, u) = (i[P-\lambda, A] u, u) = i(Au, (P-\lambda) u) - i((P-\lambda)u, A^* u).
\]
Here and below, we write $(\,\cdot\,, \,\cdot\,)$ and $\norm{\,\cdot\,}$ for the inner product and norm on $L^2(\mR^n)$.
If $(i[P,A]u, u)$ is sufficiently positive so that it controls weighted versions of $Au$ and $A^* u$, we can use Young's inequality with $\eps > 0$ in the form $2\abs{(Au, (P-\lambda)u)} \leq \eps \norm{Au}^2 + \eps^{-1} \norm{(P-\lambda)u}^2$ (also with suitable weights) to obtain a resolvent type estimate with $(P-\lambda)u$ on the right.

To motivate the choice of the conjugate operator $A = a_k(x) D_k + b(x)$, we note that $P$ and $A$ have principal symbols $p(x,\xi) = \abs{\xi}^2$ and $a(x,\xi) = a_k(x) \xi_k$.  Notice that we are omitting summation symbols over repeated indices (we will continue to use this convention in the rest of the paper). Then the commutator $i[P,A]$ has principal symbol given by the Poisson bracket 
\[
\{ p, a \} = \nabla_{\xi} p \cdot \nabla_x a - \nabla_x p \cdot \nabla_{\xi} a = 2 \partial_j a_k(x) \xi_j \xi_k.
\]
We want the last quantity to be suitably positive. If we choose $a_k = 2 \partial_k \varphi$ for some function $\varphi$, which happens for instance if $A = i[P,\varphi]$, then $\{ p, a \} = 4 \varphi''(x) \xi \cdot \xi$ where $\varphi''$ is the Hessian matrix of $\varphi$. Thus, if $\varphi$ is a suitable convex function, we expect that $A = i[P,\varphi]$ could have the required properties.

We  consider the long range magnetic perturbation  $P + L + V^L$, where 
\begin{equation} \label{eq:L}
L u := 2W^L \cdot Du + (D \cdot W^L)u,
\end{equation}
and where  $W^L $ and  $V^L$ satisfy conditions \eqref{longrange1}.

\begin{Lemma} \label{lemma_positive_commutator_first}
Let $\varphi \in C^4(\mR^n)$ be real valued, and let 
\[
A := i[P, \varphi] = 2 \partial_j \varphi D_j - i(\Delta \varphi).
\]
Let also  $z \in \mC$. Then, for any $u \in C^{\infty}_c(\mR^n)$ one has 
\begin{multline} \label{positive_commutator_identity_first}
 4(\varphi'' Du, Du) - ( (\Delta^2 \varphi) u, u ) \\
  = -2\Im(Au, (P+L +   V^L -z)u) + 4 \Im(z) (Au,u) - (i[L+  V^L,A]u, u).
\end{multline}
\end{Lemma}
\begin{proof}
With the given choice of $A$, we compute 
\begin{align*}
i[P,A]u &= -i\Delta(2 \partial_j \varphi D_j u - i(\Delta \varphi) u) + i (2 \partial_j \varphi D_j - i(\Delta \varphi))(\Delta u) \\
 &= -2i \Delta \p_j \phi D_j u - 4i \p_{jk} \phi D_j \p_k u - (\Delta^2 \phi) u - 2 \p_j \Delta \phi \p_j u \\
 &= 4 D_k (\p_{jk} \varphi D_j u) - (\Delta^2 \phi) u.
\end{align*}
Note that $(Au, v) = (u, Av)$ for $u, v \in C^{\infty}_c(\mR^n)$. The result follows since 
\begin{align*}
&4(\varphi'' Du, Du) - ( (\Delta^2 \varphi) u, u ) = (i[P,A]u, u) \\
 &= (i[P+L+  V^L-z,A]u, u) - (i[L+  V^L,A]u, u),
\end{align*}
and 
\begin{multline*}
(i[P+L+  V^L-z,A]u, u)   = i(Au, (P+L+  V^L-\bar{z})u) \\
 - i((P+L+  V^L-z)u, Au). \qedhere
\end{multline*}
\end{proof}

Note that the left hand side of \eqref{positive_commutator_identity_first} is independent of $z$. To describe the dependence on $\Re(z)$, we need the following lemma.

\begin{Lemma} \label{lemma_positive_commutator_second}
Let $\eta \in C^2(\mR^n)$ be real valued, and let $z \in \mC$. Then for any $u \in C^{\infty}_c(\mR^n)$, 
\begin{multline} \label{eq:rez_usquared}
2 \Re(z) \int \eta |u|^2 = - \int (\Delta \eta) |u|^2 + 2 \int \eta \abs{\nabla u}^2 \\
 - \int \eta ( ((P-z)u) \bar{u} + u(P - \bar{z}) \bar{u} ).
\end{multline}
\end{Lemma}
\begin{proof}
This is a direct integration by parts:
\begin{align*}
- \int (\Delta \eta) |u|^2 &= -\int \eta \Delta(u \bar{u}) \\
 &= -2 \int \eta |\nabla u|^2 - \int \eta ((\Delta u) \bar{u} + u \overline{\Delta u}).
\end{align*}
The result follows by writing $\Delta u = (\Delta+z)u - zu$ in the last terms on the right.
\end{proof}

We can combine the previous lemmas with suitable choices of $\phi$ and $\eta$ to obtain an a priori estimate for the long range magnetic resolvent. 
We need a weight $\phi$ with a large positive Hessian, so that we can later absorb certain terms on the left hand side  of \eqref{positive_commutator_identity_first}. See Remark \ref{remark_weight_motivation} below for motivation for the choice of $\varphi$.

Let $\tau\ge 1$ and write $r = \abs{x}$. We consider a radial weight
\[  \varphi  : = \tilde \phi e^{\tau \psi}\]
with $\psi(x) = \psi_0(r)$ and $\tilde \phi(x) = \tilde \phi_0(r)$, for appropriate choices of $\tilde \phi_0$ and $\psi_0$. 
The  Hessian of $\phi$ satisfies
\[ \phi'' = e^{\tau \psi}  \left( \tilde \phi'' + \tau \nabla \tilde \phi \otimes \nabla \psi + \tau \nabla \psi \otimes \nabla \tilde \phi   + \tau^2 \tilde \phi \nabla \psi \otimes \nabla \psi  + \tau \tilde \phi \psi'' \right). \]
We are going to choose $\tilde \phi(x) =\tilde \phi_0(r) = \br{r}$, where $r=|x|$. Note that the Hessian of $\tilde{\phi}$ is positive semidefinite, i.e.\ $\tilde \phi''(x) \ge  0$. Also, let us take $\psi(x)= \psi_0(r) = 1- \br{r}^{1-2\delta}$. Writing explicitly the Hessian of $\psi$ in terms of the derivatives of $\psi_0(r)$, yields
\[
(\psi''(x) \nabla u, \nabla u)=(\psi_0''(r)\partial_r u, \partial_r u)+ (\frac{\psi_0'(r)}{r}\nabla^{\perp} u, \nabla^{\perp} u),
\]
where $\nabla^{\perp} u:=\nabla u-(\nabla u \cdot \hat{x}) \hat{x}$ with $\hat{x} = x/\abs{x}$, so that $\abs{\nabla u}^2 = (\p_r u)^2 + \abs{\nabla^{\perp} u}^2$ (notice that  $\psi''$ denotes the Hessian matrix of $\psi$ and $\psi_0'' = \frac{d^2 \psi_0}{d\, r^2}$). Thus, using the condition $\tilde{\varphi}'' \geq 0$, we get
\begin{multline} \label{eq:hess1}
 (\phi''(x)\nabla u,\nabla u) \geq   2\tau ( e^{\tau \psi} \tilde \phi_0'\psi_0'\partial_r u,  \partial_r u) + \tau^2 (e^{\tau \psi} \tilde \phi_0 (\psi_0')^2 \partial_r u,  \partial_r u)  \\ 
 +  \tau ( e^{\tau \psi} \tilde \phi_0 \psi_0'' \partial_r u,\partial_r u) + \tau ( e^{\tau \psi} \tilde \phi_0 \frac{\psi_0'}{r} \nabla^{\perp} u,\nabla^{\perp} u).
\end{multline}
We remark that, since $\phi$ and $\psi$ are increasing, all the terms in the right hand side of \eqref{eq:hess1} are positive, except for the third one.
\begin{Lemma} \label{lemma_properties_wheight0}
Let $1/2<\delta<1$, and let $\varphi(x)  = \br{x} e^{\tau \psi}$, with $\psi(x) = \psi_0(r) = 1-\br{r}^{1-2\delta}$ and $\tau \ge 1$. Then, if $\beta = 2(1-\delta)(2\delta-1)$  we have that
\begin{equation} \label{eq:hess2}
 (\phi''(x) \nabla u,\nabla u) \geq  \beta \tau  \norm{ e^{\frac{1}{2}\tau \psi}  w_1 \nabla u}^2  + \tau^2 \norm{e^{\frac{1}{2}\tau \psi} w_2 \partial_r u}^2,  
\end{equation}
where $w_1^2(x)= \br{x}^{-2\delta}$, and $w_2^2(x)= \br{r} (\psi_0'(r))^2$ with $r=|x|$. Moreover, if $\alpha$ is a multi-index such that $|\alpha| =N$, then 
\begin{equation} \label{eq:wheight_derivatives}
    |\partial^{\alpha} \phi (x)| \le C_{N}\tau^{N} \br{x}^{1-N}e^{\tau \psi}, 
\end{equation}
\end{Lemma}

Since the proof is a straightforward computation, we leave it to Appendix \ref{sec_appendix}. As a consequence  \eqref{eq:wheight_derivatives} we get that
\begin{equation} \label{eq:lap_and_bilap}
|\Delta \phi(x)| \le C\tau^2  \br{x}^{-1} e^{\tau \psi} , \qquad |\Delta^2 \phi(x)| \le C\tau^4  \br{x}^{-3}e^{\tau \psi} ,
\end{equation}
which will be used in the proof of the following lemma. 
We state now the main Carleman type estimate for the long range magnetic resolvent.

\begin{Lemma} \label{lemma_longrange_resolvent_carleman} 
Let $\sigma >0$, $1/2<\delta<\infty$, and $\tau> 6 \beta^{-1}$. Let $V^L \in L^\infty(\mR^n, \mR)$ and $W^L \in L^\infty(\mR^n, \mR^n)$ satisfy \eqref{longrange2} as well as such that $\nabla \cdot W^L \in L^2_{loc}(\mR^n)$. Then  for any $v \in C^{\infty}_c(\mR^n)$, 
\begin{multline} \label{eq:Lrange_est}
\Re(z) \norm{ v}_{L^2_{-\delta}}^2 + \norm{ \nabla v}_{L^2_{- \delta}}^2   \\
 \leq 64 \beta^{-1} \tau^{-1}\norm{e^{\frac{1}{2}\tau\psi}(P+L+  V^L-z)e^{-\frac{1}{2}\tau \psi} v}_{L^2_{ \delta}}^2  + C(\beta,\tau) |\Im(z)| \Re(z)^{1/2}  \norm{v}^2,
\end{multline}
whenever $z \in \mC$ with $\abs{\Im(z)} \leq 1$ and $\Re(z) \geq C(n,\delta,\tau, W^L,   V^L) \ge 1$. 
\end{Lemma}
The condition  $\nabla \cdot W^L \in L^2_{loc}(\mR^n)$ is only necessary so that the right hand side of \eqref{eq:Lrange_est} is well defined.

\begin{Remark} \label{remark_weight_motivation}
In \cite[Section 6]{RodnianskiTao}, a similar estimate is proved by using   a commutator argument
 with the weight $|x| - \br{x}^{2-2\delta}$. To avoid problems at the origin   one can use the analogous smooth weight $ \phi = \br{x} - \br{x}^{2-2\delta}$. This weight has a positive Hessian and satisfies $\Delta^2 \phi > 0$ for $n \geq 3$, which leads to a satisfactory  estimate for the long range perturbation. The same estimate can be obtained in dimension $n=2$ using that for this choice of weight one has $|\Delta^2 \phi| \le C\br{x}^{-3}$, so that the bilaplacian term in \eqref{positive_commutator_identity_first} can be controlled appropriately. 
 
 Unfortunately, these weights do not allow one to later  absorb a large short range $W^S$ (unless $W^S$ is continuous and $\nabla \cdot W^S$ is short range). To deal with the general case we want the Hessian of $\phi$ to be as large as needed. This motivates the choice of a weight with a large parameter $\tau$. The exponential weight $\phi = \br{x}e^{\tau \psi}$   works satisfactorily since it has a large positive Hessian, as shown in Lemma \ref{lemma_properties_wheight0}.  The main difference with \cite[Section 6]{RodnianskiTao} is that now this choice leads naturally to the Carleman type estimate \eqref{eq:Lrange_est}. This kind of estimate is in line with the results in \cite{Vodev}. We have chosen $\psi$ to be bounded above and below so that the exponentials can be removed later from the estimates.

 Therefore, thanks to the choice of our weight, we have the factor $\tau^{-1}$ in the right hand side of  \eqref{eq:Lrange_est}. As we have already mentioned, this will be helpful in dealing with a general short range magnetic potential. On the other hand, the dependence  on  $\tau$ of the error term on the right hand side  is not relevant because of the factor $\Im(z)$. Once the short range potentials will  be introduced, we will fix the value of $\tau$, and then  the factor $\Im(z)$  will lead to an estimate for the error term involving the whole magnetic operator, by combining Young's inequality with the identity:
 \[
 \Im(z)\norm{u}^2=-\Im ((H-z)u, u).
 \]
 \end{Remark}
 
 \section{The Long Range estimate} \label{sec:longrange}
 
 We now prove Lemma \ref{lemma_longrange_resolvent_carleman} using the commutator method introduced in Section \ref{sec:commutator}. For the proof of this estimate, it will be useful  to state the following lemma that we will prove at the end of this section.
\begin{Lemma} \label{lemma_positive_commutator_fourth}
 Let $M:= \tau^3 C(n,\delta)( \norm{\br{x}^{1+\sigma} \partial_r W^L}_{L^\infty} + \norm{ \br{x}^\sigma W^L}_{L^{\infty}}) $. Then for any $u \in C^{\infty}_c(\mR^n)$ we have
\begin{equation} \label{eq:comm_long}
 \abs{i([L,A]u, u)}  \leq   \norm{e^{\frac{1}{2}\tau \psi}  w_1 \nabla u}^2 +  \left ( M^2+2 \right )\norm{e^{\frac{1}{2}\tau \psi}  w_1  u}^2.
\end{equation}
\end{Lemma}

We can now prove estimate \eqref{eq:Lrange_est}. The proof starts from \eqref{positive_commutator_identity_first}. Basically one hopes to be able to bound the terms on the right with an appropriate  norm of $(P + L + V^L-z) u$, or to absorb them in the left hand side using the term $(\phi'' D u,D u)$. We are helped by Lemma \ref{lemma_positive_commutator_second}, which will be used to introduce in the left the ``large" term $\Re(z)\int \eta |u|^2$, for suitable $\eta$, and which has the appropriate dependence of the estimate with $\Re(z)$. In the estimate we need to be careful to follow the dependence of the constants on the large parameters $\tau$ and $\Re(z)$.

\begin{proof}[Proof of Lemma $\ref{lemma_longrange_resolvent_carleman}$]
Notice that if \eqref{eq:Lrange_est} holds for one $\delta$, then it also hold with the same constant for every  $\delta'>\delta$. Therefore, without loss of generality, we consider $1/2<\delta< \min\left\{ (\sigma +1) /2,1 \right\}$.  Then,   Lemmas \ref{lemma_positive_commutator_first}  and \ref{lemma_properties_wheight0} give that 
\begin{multline*}
4  \beta \tau  \norm{ e^{\frac{1}{2}\tau \psi}  w_1 \nabla u}^2  + 4\tau^2 \norm{e^{\frac{1}{2}\tau \psi} w_2 \partial_r u}^2 \leq ( (\Delta^2 \phi) u, u)\\
- 2\Im(Au, (P+L +  V^L -z)u) + 4 \Im(z) (Au,u) - (i[L+  V^L,A]u, u).
\end{multline*}
Now, since $\phi$ is a radial function and $\partial_r \phi \ge 0$,
\[
-(i[  V^L,A]u,u)=  2  (\partial_r \varphi \partial_r  V^L u, u)  \le   2  (\partial_r \varphi (\partial_r  V^L)_+ u, u),
\]
where $(f)_+$ is the nonnegative part of $f$. Hence, using that $w_1^{-2} \le \br{x}^{1+\sigma}$ and that $\partial_r \varphi \le 2\tau e^{\tau\psi}$,  we have
\begin{multline} \label{eq:est1}
4  \beta \tau  \norm{ e^{\frac{1}{2}\tau \psi}  w_1 \nabla u}^2  + 4\tau^2 \norm{e^{\frac{1}{2}\tau \psi} w_2 \partial_r u}^2 \leq ( (\Delta^2 \phi) u, u)  \\
 +|2\Im(Au, (P+L +  V^L -z)u)|  + 4 |\Im(z) (Au,u)| + |(i[L,A]u, u)| \\
 +4\tau \norm{\br{x}^{1+\sigma} (\partial_r  V^L)_+}_{L^\infty}\norm{e^{\frac{1}{2}\tau \psi}w_1 u}^2.
\end{multline}
We now apply  Lemma \ref{lemma_positive_commutator_second} with $ \eta = w_1^2e^{\tau \psi}$. Using that $|\nabla w_1| \le w_1$, $|\Delta w_1^2| \le 2 n w_1^2$, and that the derivatives of $\psi$ are bounded, one can see that $|\Delta(w_1^2e^{\tau \psi})| \le C\tau^2w_1^2e^{\tau \psi} $ for suitable $C$ depending on the dimension. We use  this fact in \eqref{eq:rez_usquared} and  multiply the resulting inequality by $\beta \tau $. This yields
\begin{multline}  \label{eq:est2}
  (2\beta \tau \Re(z)  - C \tau^3) \norm{e^{\frac{1}{2}\tau \psi}w_1 u}^2 \le 2\beta \tau \norm{  e^{\frac{1}{2}\tau \psi}  w_1 \nabla u}^2 \\ +  2 \beta \tau |\Re ( (P-z)u, e^{\tau \psi} w_1^2 u)|.
\end{multline}
Adding estimates \eqref{eq:est1} and \eqref{eq:est2}  and moving two terms to the left hand side, we get
\begin{multline*}
 2 \beta \tau \norm{ e^{\frac{1}{2}\tau \psi}  w_1 \nabla u}^2  + 4\tau^2 \norm{e^{\frac{1}{2}\tau \psi} w_2 \partial_r u}^2 +   \left( 2\beta \tau  \Re(z) -K_1 \right ) \norm{ e^{\frac{1}{2}\tau \psi} w_1 u}^2 \\
  \leq  |( (\Delta^2 \phi) u, u)|  + |(i[L,A]u, u)|  +2|\Im(Au, (P+L+  V^L-z)u)|   \\
   + 2|\Im(z) (Au, u)| +  2 \beta \tau |(  (P-z)u, e^{\tau \psi} w_1^2 u)|,
\end{multline*}
where $K_1 =   C \tau^3 + 4 \tau \norm{\br{x}^{1+\sigma} (\partial_r  V^L)_+}_{L^\infty}$.
Then, Lemma \ref{lemma_positive_commutator_fourth} and estimate \eqref{eq:lap_and_bilap} allow us to absorb the first two terms on the right into the left hand side. Thus
\begin{multline*}
 (2  \beta \tau - 1)  \norm{ e^{\frac{1}{2}\tau \psi}  w_1 \nabla u}^2  + 4\tau^2 \norm{e^{\frac{1}{2}\tau \psi} w_2 \partial_r u}^2 +   \left( 2\beta \tau  \Re(z) - K_2 \right ) \norm{ e^{\frac{1}{2}\tau \psi} w_1 u}^2 \\
  \leq    2\abs{(Au, (P+L+  V^L-z)u)} + 2\abs{\Im(z) (Au, u)}   + 2 \beta \tau | ( (P-z)u, e^{\tau \psi} w_1^2 u)| \, .
\end{multline*}
where $K_2 = K_1 + C\tau^4 + M^2 + 2$. 

Since $ w_1 \leq 1 \leq w_1^{-1}$, the last term on the right hand side satisfies
\begin{multline*}
 2\beta \tau | ( (P-z)u, e^{\tau \psi} w_1^2 u)| \leq  2\beta \tau \norm{w_1^{-1} e^{\frac{1}{2} \tau \psi}(P+L+  V^L-z)u} \norm{e^{\frac{1}{2} \tau \psi}w_1 u} \\ 
 +   \norm{ e^{\frac{1}{2}\tau \psi} w_1 \nabla u}^2  
 + 2\beta\tau(2\beta\tau\norm{W^L}_{L^{\infty}}^2 + (2 + \tau )\norm{W^L}_{L^{\infty}}  +\norm{  V^L}_{L^{\infty}}) \norm{e^{\frac{1}{2}\tau \psi} w_1 u}^2,
\end{multline*}
where we have used Young's inequality and the fact that
\begin{multline*}
|((2W^L \cdot \nabla u + \nabla \cdot W^L u) ,e^{\tau \psi}w_1^2 u) | \\
= | ( W^L u ,\nabla(  e^{\tau \psi}w_1^2  u))| + |(W^L \cdot  \nabla u ,  e^{\tau \psi} w_1^2 u)| \\
\le 2\norm{W^L}_{L^{\infty}}  \norm{ e^{\frac{1}{2}\tau \psi} w_1 \nabla u} \norm{e^{\frac{1}{2}\tau \psi} w_1 u} +  (2 + \tau)\norm{W^L}_{L^{\infty}} \norm{e^{\frac{1}{2}\tau \psi} w_1 u}^2 ,
\end{multline*}
which follows by integration by parts and taking into account the inequality $|\nabla  (e^{\tau \psi} w_1^2)| \le (2 + \tau) e^{\tau \psi} w_1^2 $.
This and Young's inequality  lead to the estimate 
\begin{multline}
 (2  \beta  \tau -2)  \norm{ e^{\frac{1}{2}\tau \psi}  w_1 \nabla u}^2  + 4\tau^2 \norm{e^{\frac{1}{2}\tau \psi} w_2 \partial_r u}^2 +   \left( 2\beta \tau  \Re(z) - K_3 \right ) \norm{ e^{\frac{1}{2}\tau \psi} w_1 u}^2 \\
  \leq   \norm{w_1^{-1} e^{\frac{1}{2} \tau \psi}(P+L+  V^L-z)u}^2 +  2\abs{\Im(z) (Au, u)} \\
   +   2\abs{(Au, (P+L+  V^L-z)u)}, \label{commutator_intermediate}
\end{multline}
where  now $K_3 = K_2 + 2\beta\tau(2\beta\tau\norm{W^L}_{L^{\infty}}^2 + (2 + \tau )\norm{W^L}_{L^{\infty}}  +\norm{  V^L}_{L^{\infty}} ) + \beta^2 \tau^2 $.

Again, we estimate the last term on the right of \eqref{commutator_intermediate}. We have that
\begin{align*}
 \abs{ w_1 A u} &\le 2 w_1 e^{\tau \psi } (r\br{r}^{-1} +  \tau \br{r} \psi_0'(r)) |\partial_r u| +  w_1 |\Delta \phi ||u| \\
 &\le 2e^{\tau \psi } (w_1 +  \tau w_2)  |\partial_r u| +  C\tau^2 e^{\tau \psi }\abs{w_1 u} 
\end{align*}
where we have used \eqref{eq:lap_and_bilap} and that $w_1^2\br{r} =\br{r}^{1-2\delta}\le 1 $ (recall that $w_2^2 = \br{r}\psi_0'(r)^2$). Then, we can apply several times Young's inequality with  suitable values of $\eps$, to obtain   
\begin{multline*}
 2\abs{(w_1Au, w_1^{-1}(P+L+  V^L-z)u)} \le  4\norm{ e^{\frac{1}{2}\tau \psi } w_1 \nabla u}^2 + 4 \tau^2\norm{ e^{\frac{1}{2}\tau \psi } w_2 \partial_r u}^2\\
 + C^2\tau^4 \norm{ e^{\frac{1}{2}\tau \psi } w_1  u}^2 + 8 \norm{w_1^{-1} e^{\frac{1}{2} \tau \psi}(P+L+  V^L-z)u}^2.
 \end{multline*}
 With these choices, returning to the estimate \eqref{commutator_intermediate},  the $4 \tau^2\norm{ e^{\frac{1}{2}\tau \psi } w_2 \partial_r u}^2$ terms can be cancelled out. This yields 
\begin{multline} \label{eq:step}
  ( 2  \beta \tau -6 )  \norm{ e^{\frac{1}{2}\tau \psi}  w_1 \nabla u}^2  +   \left( 2\beta \tau \Re(z) - K_3-C^2\tau^4 \right ) \norm{ e^{\frac{1}{2}\tau \psi} w_1 u}^2 \\
  \leq   16 \norm{w_1^{-1} e^{\frac{1}{2} \tau \psi}(P+L+  V^L-z)u}^2  +  2\abs{\Im(z) (Au, u)}.
\end{multline}

In the estimate \eqref{eq:step}, the last term on the right hand side still needs to be controlled. This term cannot be absorbed directly in the left hand side, since $|\nabla \phi|$ does not have any decay at infinity. We will just estimate this term in such a way that it yields the last term on the right of \eqref{eq:Lrange_est}. Using that $\abs{\nabla \varphi} \leq  (1+\tau)e^{\tau}: =a(\tau)$, Young's inequality yields that
\begin{multline} \label{im_au_u_estimate}
 \abs{\Im(z) (Au, u)} = \abs{\Im(z) ( (\nabla  \varphi \cdot D u, u) + (u, \nabla  \varphi \cdot D u))} \\
 \leq 2 a(\tau) \abs{\Im(z)}  \norm{\nabla u} \norm{u} 
\le  \abs{\Im(z)} \left ( \frac{1}{ \Re(z)^{1/2}} \norm{\nabla u}^2 +   a(\tau)^2\Re(z)^{1/2} \norm{u}^2 \right ) .
\end{multline} 
We now estimate $ \norm{\nabla u}$ as follows:
\[
\norm{\nabla u}^2 = (Pu, u) = ((P-z)u, u) + \Re(z) \norm{u}^2 + i \,\Im(z) \norm{u}^2.
\]
Taking the real part and adding and subtracting the long range potentials  leads to
\begin{align*}
\norm{\nabla u}^2 &= \Re((P-z)u, u) + \Re(z) \norm{u}^2 \\
 &= \Re((P+L+  V^L-z)u, u) - ((L+  V^L)u, u) + \Re(z) \norm{u}^2.
\end{align*}
Integrating by parts in the term  $((\nabla \cdot W^L) u,u)$, as we did previously, gives
\[ |((L+  V^L)u, u)| \le  |(W^L u,\nabla u)| + |(W^L \cdot \nabla u, u)| + \norm{V^L}_{L^\infty}\norm{u}^2 ,\]
and hence, using Young's inequality and taking $\Re(z)> 2\norm{W^L}_{L^{\infty}}^2  +\norm{  V^L}_{L^{\infty}}$, yields 
\begin{multline*}
\frac{1}{2} \norm{\nabla u}^2 \leq  \abs{((P+L+  V^L-z)u, u)} \\  
+ ( 2\norm{W^L}_{L^{\infty}}^2  +\norm{  V^L}_{L^{\infty}} + \Re(z)) \norm{u}^2 \\
\le   \abs{((P+L+  V^L-z)u, u)}   
   + 2 \Re(z) \norm{u}^2.
\end{multline*}
Inserting this in \eqref{im_au_u_estimate} and  using that  $|\Im(z)| \le 1 \le \Re(z)$  we get that
\begin{equation*}
 \abs{\Im(z) (Au, u)}  
\le   2 \abs{((P+L+  V^L-z)u, u)}   \\
   + C(\tau)|\Im(z)| \Re(z)^{1/2}  \norm{u}^2.
\end{equation*}

Therefore, returning to \eqref{eq:step} and using the previous fact together with $e^{\frac{1}{2}\tau \psi}  \ge 1$, the resulting estimate is 
\begin{multline*} 
 ( 2  \beta \tau -6 )   \norm{ e^{\frac{1}{2}\tau \psi}  w_1 \nabla u}^2  +   \left(  2\beta \tau \Re(z) - K_4 \right ) \norm{ e^{\frac{1}{2}\tau \psi} w_1 u}^2 \\
  \leq   32 \norm{w_1^{-1} e^{\frac{1}{2} \tau \psi}(P+L+  V^L-z)u}^2  + C(\tau) |\Im(z)| \Re(z)^{1/2}  \norm{u}^2, 
\end{multline*}
where $ K_4= K_3 +C^2\tau^4 +1$. 
To finish, we use the fact that
\[
\nabla(e^{\frac{1}{2}\tau \psi}  u) = e^{\frac{1}{2}\tau \psi} \nabla u + \frac{1}{2}\tau e^{\frac{1}{2}\tau \psi}(\nabla \psi) u,
\]
in the first term on the left, and we fix $v=  e^{\frac{1}{2}\tau \psi} u$. Then, since $|\nabla \psi| \le 1$, we obtain that 
\[
\norm{w_1 \nabla v}^2 \leq 2 \norm{e^{\frac{1}{2}\tau \psi} w_1 \nabla u}^2 + \tau^2 \norm{e^{\frac{1}{2}\tau \psi} w_1 u}^2,
\]
and consequently
\begin{multline*} 
 ( \beta \tau-3) \norm{   w_1 \nabla v}^2  +   (  2\beta \tau \Re(z) - K_4 -(\beta \tau - 3) \tau^2) \norm{  w_1 v}^2 \\
  \leq   32  \norm{w_1^{-1} e^{\frac{1}{2} \tau \psi}(P+L+  V^L-z) e^{-\frac{1}{2} \tau \psi} v}^2 + C(\tau) |\Im(z)| \Re(z)^{1/2}  \norm{v}^2, 
\end{multline*}
where in the last term we have used again that $e^{\frac{1}{2}\tau \psi}\ge 1$. We choose now $\tau>6 \beta^{-1}$, so that $\beta \tau - 3 \geq \frac{1}{2}  \beta \tau$, and $\Re(z) > \tau^{-1} \beta^{-1}(K_4 + (\beta \tau - 3) \tau^2)$. This yields
\begin{multline*} 
 \norm{   w_1 \nabla v}^2  +    \Re(z) \norm{  w_1 v}^2 \\
  \leq   \frac{64}{\beta \tau}  \norm{w_1^{-1} e^{\frac{1}{2} \tau \psi}(P+L+  V^L-z) e^{-\frac{1}{2} \tau \psi} v}^2 + C(\tau,\beta) |\Im(z)| \Re(z)^{1/2}  \norm{v}^2, 
\end{multline*}
 which  proves the claim.
\end{proof}

Finally, we prove the estimate \eqref{eq:comm_long}.  Here the long range conditions on the potentials play an essential role.

\begin{proof}[Proof of Lemma \ref{lemma_positive_commutator_fourth}]
A direct computation shows that 
\begin{multline*}
i[L,A]u = 4 (\phi'' W^L - (\nabla W^L) \nabla \phi) \cdot Du  
+ 2i(\nabla \varphi \cdot \nabla  (\nabla \cdot W^L) - W^L \cdot (\nabla \Delta \phi))u
 \end{multline*}
where $\nabla W^L$ is the Jacobian matrix of $W^L$. We are going to use \eqref{eq:wheight_derivatives} several times.
 We begin by studying the first and last terms. Since $2\delta-1 < \sigma$, we get
\begin{multline}  \label{eq:LA1}
|4(\phi'' W^L \cdot D u,u) - 2i(W^L \cdot (\nabla \Delta \phi) u,u)| \\ 
\le C \tau^{3} \norm{ \br{x}^\sigma W^L}_{L^{\infty}} (\norm{e^{\frac{1}{2}\tau \psi} w_1 \nabla u} \norm{e^{\frac{1}{2}\tau \psi} w_1 u} + \norm{e^{\frac{1}{2}\tau \psi} w_1 u}^2).
\end{multline}
Also, since $\phi$ is radial, $\nabla W^L(\nabla \phi) = (\partial_r \phi) \partial_r W^L  $, which implies
\begin{equation}
 |( (\nabla W^L) \nabla \phi \cdot \nabla u ,u) | \le   C\tau \norm{\br{x}^{1+\sigma} \partial_r W^L}_{L^\infty}  \norm{e^{\frac{1}{2}\tau \psi} w_1 \nabla u} \norm{ e^{\frac{1}{2}\tau \psi}w_1 u}.
\end{equation}
Let us estimate the remaining term. By the Leibniz rule  we have that
\begin{multline*}
    \nabla \varphi \cdot \nabla (\nabla \cdot W^L) u = \nabla \cdot(( \nabla  W^L) (\nabla \phi) u)\\  - \nabla W^L (\nabla \phi) \cdot \nabla u  - (\partial_j \partial_k \phi) \partial_kW^L_j  u.
\end{multline*}
We expand again the last term, so that we only have terms with radial derivatives of the magnetic potential, 
\[
    (\partial_j \partial_k \phi) \partial_kW^L_j  u = \nabla \cdot (\phi'' (W^L) u) - (\nabla \Delta \phi) W^L u - \phi''(W^L) \cdot \nabla u,
\]
where recall that we denote the Hessian of $\phi$ by $\phi''$. Hence, integrating by parts the first term on the right hand side of  the previous two expressions  yields  
\begin{multline*}
    |((\nabla \varphi \cdot \nabla ( \nabla \cdot W^L) u,u)| \le  
   C \tau^{3} (\norm{\br{x}^{1+\sigma} \partial_r W^L}_{L^\infty}  +  \norm{ \br{x}^\sigma W^L}_{L^{\infty}} )  \times \dots \\
    (\norm{e^{\frac{1}{2}\tau \psi} w_1 \nabla u} \norm{ e^{\frac{1}{2}\tau \psi} w_1u} + \norm{ e^{\frac{1}{2}\tau \psi} w_1 u}^2).
\end{multline*} 
Then, writing $M:= C \tau^3  ( \norm{\br{x}^{1+\sigma} \partial_r W^L}_{L^\infty} + \norm{ \br{x}^\sigma W^L}_{L^{\infty}}) $ for an appropriate  constant $C= C(n,\delta)$, one has 
\begin{align*}
 \abs{i([L,A]u, u)} &\leq 2M(\norm{ e^{\frac{1}{2}\tau \psi} w_1 \nabla u} \norm{ e^{\frac{1}{2}\tau \psi} w_1u} + \norm{e^{\frac{1}{2}\tau \psi} w_1 u}^2) \\
 &\leq  \norm{ e^{\frac{1}{2}\tau \psi} w_1 \nabla u}^2 +  \left ( M^2+2 \right )\norm{ e^{\frac{1}{2}\tau \psi} w_1 u}^2. \qedhere
\end{align*}

\end{proof}

\section{Shifting the long range estimate to $H^{-1}$} \label{sec:shift}

Even if we have the help of the large parameters $\tau$ and $\Re(z)$ in  \eqref{eq:Lrange_est}, we cannot introduce directly the  short range potentials in the right hand side. This is due to the fact that $\nabla \cdot W^S$ is not necessarily an $L^\infty(\mR^n)$ function under the  condition \eqref{shortrange1} assumed on the potentials. That is, the short range perturbation is not bounded as an operator from $H^1$ to $L^2$. However, it is bounded from $H^1$ to $H^{-1}$ as it was pointed out in \cite{Vodev}.  To overcome this difficulty, we are going to derive a better version of estimate \eqref{eq:Lrange_est}, now from $H^1$ to $H^{-1}$. Of course, this is not necessary when assuming the extra condition \eqref{shortrange2}. In this case, we will just  show  by analogous methods that \eqref{eq:Lrange_est} can be improved to an estimate from  $H^2$ to $L^2$.

From now on, just to simplify notation,  we switch to the conventions of semiclassical analysis.
\begin{Definition}
Let $k$ be a nonnegative integer. We define the space $H^{k}_{scl}(\mR^n):=H^{k}_{scl}$ as the $H^k(\mathbb{R}^n)$-Sobolev space with semiclassical parameter $h>0$, equipped with the norm
\[
\left \| u \right \|_{H^{k}_{scl}}^2 = \sum_{\left| \alpha \right|\leq k} \norm{h^{\left| \alpha \right|}\partial^\alpha u}^2.
\]
We also consider its dual space $H^{-k}_{scl}(\mathbb{R}^n):=H^{-k}_{scl}$ with norm given by 
\begin{equation*}\label{cutoff1}
\left \| u \right \|_{H^{-k}_{scl}}=\underset{ \vartheta  \, \in\, C^{\infty}_0(\mathbb{R}^n)\setminus \left\{ 0 \right\}}{\sup} \dfrac{\left | \left \langle u, \vartheta  \right \rangle_{\mathbb{R}^n} \right |}{\left \| \vartheta  \right \|_{H^{k}_{scl}}},
\end{equation*}
where $\left \langle \, \cdot \,, \,\cdot \, \right \rangle_{\mathbb{R}^n}$ denotes the distribution duality in $\mathbb{R}^n$. 
\end{Definition}
In our estimates, the semiclassical structure emerges naturally by taking $h= \Re(z)^{-1/2}$, so that we now have $z=h^{-2}+i \Im (z)$.
Also, let us write  $w(x) =  w_1(x) = \br{x}^{-\delta}$.  Under this framework, \eqref{eq:Lrange_est} can be written as follows:
\begin{equation}\label{eq:Lrange_est_semicl}
\norm{w v }_{H^1_{scl}}^2  \le   64 \beta^{-1} \tau^{-1} h^{2} \norm{w^{-1} e^{\frac{1}{2}\tau \psi} (P+L+V^L-z) e^{-\frac{1}{2}\tau \psi} v}^2  + bh\norm{v}^2,
\end{equation}
where  $b = C(\beta, \tau) |\Im(z)|$. Now we want to prove the following proposition.
\begin{Proposition} \label{prop_loongrange_est_shifted}
Assume that all the conditions in the statement of Lemma  \ref{lemma_longrange_resolvent_carleman} hold, and that  $\nabla \cdot W^L \in L^\infty(\mR^n)$.  Then, for any $v \in C^{\infty}_c(\mR^n)$, 
\begin{multline}\label{eq:Lrange_est_H-1}
 \norm{wv}_{H_{scl}^{1 +a}}^2 \le  C bh\norm{ v}_{H^{-1 +a}_{scl}}^2  \\ 
 +  C \beta^{-1} \tau^{-1}  h^2 \norm{w^{-1 } e^{\frac{1}{2}\tau \psi} (P+L+V^L-z) e^{-\frac{1}{2}\tau \psi} v}_{H^{-1 +a}_{scl}}^2   ,
\end{multline}
whenever $z \in \mC$ with $\abs{\Im(z)} \leq 1$, $ h <  h_0(n,\delta,\tau, W^L,   V^L) < 1$ and $a=0,1$. Here $C$ is an absolute constant.
\end{Proposition}

In the remaining part of this section, to simplify further the notation, we put 
\begin{equation} \label{eq:G_def}
G_{h, \tau}:=  h^2 e^{\frac{1}{2}\tau \psi}   (P + L + V^L) e^{-\frac{1}{2}\tau \psi} .
\end{equation}
Recall that $L$ was defined in \eqref{eq:L}. To prove the estimate \eqref{eq:Lrange_est_H-1} in the case $a=0$, instead of commuting with the operator $\br{hD}^{-1}$ to shift estimate \eqref{eq:Lrange_est_semicl} one derivative down, we follow \cite{Vodev} and commute with a resolvent operator (in the case $a=1$ we only need to get an extra one derivative gain in \eqref{eq:Lrange_est_semicl}). In both cases, we need the following result.
\begin{Lemma} \label{lemma_shift} 
Let $h>0$, $\tau\ge 1$,  $k \in \mR$, and $a=0,1$. Consider in \eqref{eq:G_def}  $V^L\in L^\infty({\mathbb{R}^n, \mathbb{R}})$, $W^L\in L^\infty(\mathbb{R}^n, \mathbb{R}^n)$ with $\nabla \cdot W^L \in L^\infty(\mR^n)$, and $\psi \in C^{2}(\mathbb{R}^n)$   independent of $h$ and $\tau$ such that $|\nabla \psi|,|\Delta \psi|<\infty$. Then, for all $u\in H_{scl}^{-1 +a}$ we have that
\begin{equation} \label{eq:G_weights}
 \norm{w^{k}(G_{h, \tau} - i)^{-1 } w^{-k} u}_{H^{1+a}_{scl}} \le 4 \norm{ u  }_{H^{-1 +a}_{scl}} ,
\end{equation}
if $a=0,1$ and $h < c\tau^{-1}$ for a small $c=c(V^L,W^L,\psi,k)$.
\end{Lemma} 
This lemma is essentially \cite[Lemma 3.2]{Vodev}. Nonetheless, for the interested reader we give a proof in the appendix.
We now prove estimate \eqref{eq:Lrange_est_H-1}. 
\begin{proof}[Proof of Proposition \ref{prop_loongrange_est_shifted}]
To simplify the computations, throughout this proof we denote $G:=G_{h, \tau}$. Now we can combine \eqref{eq:Lrange_est_semicl} and Lemma \ref{lemma_shift}  to get \eqref{eq:Lrange_est_H-1}, using that the identity
\begin{equation} \label{eq:res_carleman3}
v = (h^2z - i) (G - i)^{-1} v  + (G - i)^{-1} (G - h^2 z ) v,
\end{equation}
holds for any $z \in \mC$. Multiplying \eqref{eq:res_carleman3}  by the weight $w$ and taking the $H_{scl }^{1+a}$ norm squared we have
\begin{equation} \label{eq:sfift_a}
\norm{ w v }_{H^{1+a}_{scl}}^2 \le 8  \norm{ w  (G-i)^{-1}     v }_{H^{1+a}_{scl}}^2  +2\norm{ w  (G-i)^{-1}   (G - h^2z)  v }_{H_{scl}^{1+a}}^2,
\end{equation}
since  $|i-h^2z| \le 2$. From here we split the proof into two cases.\\
\textit{Case $a=0$}. We have shown that \eqref{eq:Lrange_est_semicl} holds for $v\in C^\infty_c$. We can easily extend this estimate for $v\in H^2_{\delta}$. Indeed, take a sequence of functions $v_j \in C^\infty_c$ such that $v_j \to v$ in $H^2_{\delta}$. Applying \eqref{eq:Lrange_est_semicl} to the $v_j$, we can pass to the limit  using that $G$ is bounded from $H^2_\delta$ to $L^2_{\delta}$.  By Lemma \ref{lemma_shift} with $k=-1$,  $(G-i)^{-1} v \in H^2_\delta$ so by the previous density argument, we can apply \eqref{eq:Lrange_est_semicl} to the first term of \eqref{eq:sfift_a} with $a=0$. Then 
\begin{multline*} 
\norm{ w v }_{H^{1}_{scl}}^2 \le C \beta^{-1} \tau^{-1} h^{-2} \norm{ w^{-1}  (G - h^2z)  (G-i)^{-1}     v }^2 \\
   + C bh\norm{(G-i)^{-1} v}^2 +2\norm{ w  (G-i)^{-1}   (G - h^2z)  v }_{H_{scl}^{1}}^2.
\end{multline*}
  Using that the operators
$ (G - h^2z) $ and $ (G-i)^{-1} $ commute in the first  term on the right hand side  and taking $h$ small enough such that $\tau^{-1} h^{-2} \beta^{-1}>1$, yields 
\begin{multline*} 
\norm{ w v }_{H^1_{scl}}^2 \le C \beta^{-1} \tau^{-1} h^{-2} \norm{ w^{-1}   (G-i)^{-1} (G - h^2z)    v }^2 \\
  + C bh\norm{(G-i)^{-1} v}^2 + 2 \beta^{-1} \tau^{-1} h^{-2}\norm{ w  (G-i)^{-1}   (G - h^2z)  v }_{H_{scl}^{1}}^2.
\end{multline*}
Hence, applying Lemma \ref{lemma_shift} with $k=-1$, $k=0$, and $k=1$ to each term on the right   and using that $w\le w^{-1}$ in the last term, we finally  obtain
\begin{equation*} 
\norm{ w v }_{H^1_{scl}}^2 \le C \beta^{-1} \tau^{-1} h^{-2} \norm{ w^{-1}    (G - h^2z)    v }_{H^{-1}_{scl}}^2  + C bh\norm{ v}_{H^{-1}_{scl}}^2,
\end{equation*} 
which combined with \eqref{eq:G_def} yields the desired estimate. \\
\textit{Case $a=1$}. By a straightforward computation and applying Lemma \ref{lemma_shift} with $k,a=1$ to each term on the right  of \eqref{eq:sfift_a} we get
\begin{equation*} 
\norm{ w v }_{H^{2}_{scl}}^2 \le 32 (4  \norm{ w    v }^2  +\norm{ w  (G - h^2z)  v }^2).
\end{equation*} 
Hence, using \eqref{eq:Lrange_est_semicl} in the first term, and taking  $\tau^{-1} h^{-2} > 1$ in the second gives the desired estimate, as in the previous case. 
\end{proof}

\section{Absorbing the short range potentials} \label{sec:shortrange}

In this section we finally prove Theorem \ref{thm:main1}. The first step is to introduce the short range perturbation in \eqref{eq:Lrange_est_H-1}. Once we have an estimate for the full operator, we can fix an appropriate value of $\tau$ and remove the exponential conjugation. The final step is to extend by density the resulting estimate to an appropriate functions space, so that we are not restricted to compactly supported smooth functions. Here we shall use the Friedrichs lemma. In this step we will strengthen the assumptions on the long range potentials slightly and assume that \eqref{longrange1} holds. First  recall that  $H=P + 2W\cdot D + (D \cdot W) +V$.
\begin{Proposition} \label{prop_shortrange_carleman} 
 Let  $W \in L^\infty(\mR^n, \mR^n)$ with $\nabla \cdot W^L \in L^\infty(\mR^n)$, let $V \in L^\infty(\mR^n,\mR)$, and let $\sigma >0$ be such that \eqref{shortlong}, \eqref{shortrange1}, and \eqref{longrange2} hold.  Assume also that  $1/2<\delta<\infty$, $\tau> \tau_0(\delta, W^S,V^S) \ge 1 $ and $a =0,1$. Then  for any $v \in C^{\infty}_c(\mR^n)$, 
\begin{multline} \label{eq:carleman_shortrange}
 \norm{  wv}_{H_{scl}^{1 +a }}^2 \le  \; C(\beta,\tau) |\Im(z)| h \norm{ v}^2 \\  
 + C\beta^{-1} \tau^{-1} h^2 \norm{ w^{-1}e^{{\frac{1}{2}} \tau \psi}  (H -z )  e^{-{\frac{1}{2}}\tau \psi} v}_{H_{scl}^{-1 +a }}^2 .
\end{multline}
whenever $z \in \mC$ with $\abs{\Im(z)} \leq 1$, $a=0$, and $ h   \le h_0(n,\delta,\tau, W,   V) < 1$.  Moreover, under the extra assumption \eqref{shortrange2} the previous estimate also holds for $a=1$.
\end{Proposition}
\begin{proof}
Again, we assume $1/2<\delta< \min\left\{ (\sigma +1) /2,1 \right\}$ without loss of generality. We first consider the case $a=0$. Adding and subtracting the terms with the short range perturbation in the right hand side of \eqref{eq:Lrange_est_H-1}, we have
\begin{multline*} 
 \norm{  wv}_{H_{scl}^{1}}^2    
 \le  C\beta^{-1} \tau^{-1} h^2 \norm{ w^{-1}e^{\frac{1}{2} \tau \psi}  (H -z )  e^{-\frac{1}{2} \tau \psi} v}_{H_{scl}^{-1}}^2 \\
 + C\beta^{-1} \tau^{-1}  \norm{ w^{-1} h( 2W^S\cdot D + i \tau W^S \cdot\nabla \psi + D \cdot W^S +V^S )v}_{H_{scl}^{-1}}^2  + C bh\norm{ v}_{H^{-1}_{scl}}^2.
\end{multline*} 
 As we mentioned previously, we can estimate the term $(D\cdot W^S)v$ in the $H_{scl}^{-1}$ norm, in fact, we have
 \[ w^{-1}(\nabla \cdot W^S) v =   \nabla \cdot( w^{-1} W^S v ) - W^S \cdot \nabla ( w^{-1}v ),\]
 in the sense of distributions. Thus, 
\begin{multline} \label{eq:dual_by_parts}
 \norm{w^{-1}h \nabla \cdot W^S v}_{H_{scl}^{-1}}  \le   \norm{h  \nabla \cdot( w^{-1} W^S v ) }_{H_{scl}^{-1}}   +  \norm{ W^S \cdot h\nabla ( w^{-1}v )}_{H_{scl}^{-1}}    \\
\le  \norm{ w^{-1} W^S v  }  + \norm{ w^{-2} W^S}_{L^\infty}\norm{ w^2 h\nabla ( w^{-1}v )}_{H_{scl}^{-1}}  \le C \norm{ w^{-2} W^S}_{L^\infty} \norm{wv},
 \end{multline} 
since $|\nabla w^{-1}| \le \delta w^{-1}$. Therefore, using that ${|\nabla \psi| \le 1}$ and $w^{-1}\ge 1$,   we obtain 
\begin{multline*} 
 \norm{  wv}_{H_{scl}^{1}}^2    
 \le  C\beta^{-1} \tau^{-1} h^2 \norm{ w^{-1}e^{\frac{1}{2} \tau \psi}  (H -z )  e^{-\frac{1}{2} \tau \psi} v}_{H_{scl}^{-1}}^2\\
 +  C(\beta) \norm{wv}_{H_{scl}^{1}}^2 \left( (\tau^{-1}  +  \tau  h^2 ) \norm{ w^{-2} W^S}_{L^\infty}^2   
 +   \tau^{-1}  h^2 \norm{ w^{-2} V^S}_{L^\infty}^2 \right)  + C bh\norm{ v}^2
\end{multline*} 
for an appropriate constant $C(\beta)>0$. Since $w^{-2} \le \br{x}^{1+\sigma}$, the short range conditions on the potentials guarantee that the $L^\infty$ norms appearing in the previous estimate are finite.  Hence taking $\tau > 4 C(\beta) \norm{  \br{x}^{1+\sigma} W^S}_{L^\infty}^2$ and $h^2 <  \left (4 C(\beta) \tau (\norm{  \br{x}^{1+\sigma} W^S}_{L^\infty}^2 + \norm{  \br{x}^{1+\sigma} V^S}_{L^\infty}^2 ) \right)^{-1}$ to absorb the middle term on the right in the left hand side, and using that $b= C(\delta,\tau) |\Im(z)|$ yields the desired result. 

The case $a=1$ is even more simple since we do not need the integration by parts in \eqref{eq:dual_by_parts} thanks to the fact that the norm $\norm{  \br{x}^{1+\sigma} \nabla \cdot  W^S}_{L^\infty}$ is finite by \eqref{shortrange2}.
\end{proof}

We are going to use the previous proposition to prove Theorem \ref{thm:main1}, but first we need a couple of lemmas. The first one is necessary  control the term with the factor $|\Im(z)|$ in \eqref{eq:carleman_shortrange}.
\begin{Lemma}  \label{lemma_imz_usquared}
  Let  $W \in L^\infty(\mR^n, \mR^n)$, $V \in L^\infty(\mR^n,\mR)$, and $u\in C^\infty_c(\mR^n)$. Then
\begin{equation*}
\abs{\Im(z)} \norm{u}^2 \leq \abs{((H -z)u, u)}.
\end{equation*}
\end{Lemma}
\begin{proof}
This follows by the symmetry of the operator $H$. In fact, by integration by parts,  $\Im((P+2W\cdot D + D \cdot W +V)u,u) = 0$, and therefore
\begin{equation*} 
\Im(z) \norm{u}^2 = \Im(zu, u) = - \Im((P+2W\cdot D + D \cdot W +V-z)u, u).
\end{equation*}
This proves the lemma.
\end{proof}
We now state the Friedrichs lemma as in \cite[Lemma 17.1.5]{H2} (see also \cite[Lemma 1.5.2]{cherriermilani}). We need this result so that we can remove the restriction $u \in C^\infty_c$ using a density argument. 
\begin{Lemma} \label{lemma_firedrich}
Let $v \in L^2$ and let $ |a(x)-a(y)| \le  M|x-y|$ if $x,y \in \mR^n$. If $\Phi \in C^\infty_c$ and $\Phi_\eps(x) = \Phi(x/\eps)\eps^{-n}$, then
\[ 
\norm{(a D_j v)*\Phi_\eps-a(D_jv*\Phi_\eps)}_{L^2} = o(1) \text{ as $\eps \to 0$}.
\]
\end{Lemma}
We can now prove the main result in this paper.
\begin{proof}[Proof of Theorem \ref{thm:main1}]
Let $u \in C^{\infty}_c(\mR^n)$. We fix a sufficiently large $\tau $ so that \eqref{eq:carleman_shortrange} holds, and choose $v = e^{\frac{1}{2} \tau \psi} u$. Next, we remove the exponentials by using that $1 \le e^{ \frac{1}{2}\tau \psi} \le e^{ \frac{1}{2}\tau }$ (there are some extra terms appearing in the left hand side due to the $H^{1 + a}_{scl}$ norm, but they can be absorbed easily for $h<c\tau^{-1}$). Hence the estimate 
\begin{equation*}
\norm{ w  u}_{H_{scl}^{1+a}}^2 \le    C |\Im(z)| h \norm{u}^2  + C h^2  \norm{ w^{-1}  (H -z )  u }_{H_{scl}^{-1+a}}^2
\end{equation*}
holds for  $a=0,1$, depending on the conditions assumed on $W$, and for some $C = C(\delta, V,W)>0$ (also depending on the fixed $\tau$). Then, we can apply Lemma \ref{lemma_imz_usquared} and Young's inequality to the first term on the right. Thus
\begin{equation*}
C |\Im(z)| h \norm{u}^2 \le  \frac{1}{4}\norm{wu}_{H^1_{scl}}^2
+ C^2 h^2 \norm{w^{-1}  (H -z )  u}_{H^{-1}_{scl}}^2 .
\end{equation*}
 This yields the  estimate
\begin{equation} \label{eq:Hest1}
\norm{ w  u}_{H_{scl}^{1+a}} \le C(\delta, V,W) h \norm{w^{-1}  (H -z )  u}_{H_{scl}^{-1+a}} .
\end{equation}
This estimate holds under the assumption that $u \in C^\infty_c$. We are going to extend it for any $u \in H^1_{scl}$ such that $w^{-1} (H-z)u \in H^{-1 +a}_{scl}$. We restrict ourselves to the case of $a=0$, since $a=1$ follows from the same arguments with minor modifications (the condition \eqref{shortrange2} is again essential in the case $a=1$ so that the short range terms are bounded in $L^2$ instead of just in $H^{-1}$). Also, we now drop temporarily the semiclassical spaces since all the convergence arguments we need work independently of $h$. 

 Let $\Phi_\eps(x) := \eps^{-n}\Phi(\eps^{-1} x)$, where $\Phi $ is a standard smooth mollifier, and  $\chi_\eps := \chi(\eps x)$ where $0 \le \chi(x) \le 1$ is a smooth cut-off function such that $\chi(x) =1$ for $|x| \le 1$ and $\chi(x) = 0$ if $|x|\ge 2$. Let also  $u \in H^1$, and $u_\eps = \chi_\eps(u*\Phi_\eps)$. Then $u_\eps \in C^\infty_c$ and we have that $u_\eps \to u$ in $H^1$ as $\eps \to 0$. We would like to show  that  $w^{-1} (H-z) u_\eps \to  w^{-1} (H-z) u$ in $H^{-1}$ when $\eps \to 0$. Notice that since the potentials are bounded, for any $u\in H^1$ we have  $(H-z)u \in H^{-1}$.

  We decompose $H$ in a long range Hamiltonian and a short range perturbation $ H u  = (H^L +S) u$ where 
\begin{align*}
 H^L u &= (-\Delta  +  2W^L\cdot D  + V^L)u,  \\
 Su &= (2W^S\cdot D + D\cdot W^S +  D \cdot W^L + V^S)u. 
 \end{align*}
The perturbation $S$ is bounded from $H^1$ to $H^{-1}$, in fact  a better estimate holds:
\begin{equation}\label{eq:S_bound}
 \norm{w^{-1} Su}_{H^{-1}} \le C \norm{w u }. 
\end{equation}
This follows directly from the short range conditions \eqref{shortrange1} on $W^S,V^S$ and the long range conditions \eqref{longrange1} on $W^L$ (we have already controlled the term $\nabla  \cdot W^S$ in \eqref{eq:dual_by_parts}). Therefore, it is enough to show  that  $w^{-1} (H^L-z) u_\eps \to  w^{-1} (H^L-z) u$ in $H^{-1}$ when $\eps \to 0$.

Now, let $v_\eps = u*\Phi_\eps$ so that $u_\eps = \chi_\eps v_\eps$.  Then
\begin{multline} \label{eq:o_eps1}
 \norm{  \chi_\eps w^{-1} (H^L-z)  v_\eps -  w^{-1} (H^L-z) u_\eps }_{H^{-1}} =  \norm{ w^{-1} [\chi_\eps,H^L-z] v_\eps }_{H^{-1}}   \\ 
 =   \norm{ w^{-1} ( \Delta \chi_{\eps} + 2 \nabla \chi_\eps \cdot \nabla +2i W^L\cdot \nabla \chi_\eps ) v_\eps }_{H^{-1}}   \le C\eps^{1-\delta}\norm{v_\eps}, 
\end{multline}
where we have used that $|w^{-1} \nabla \chi_\eps|,|w^{-1} \Delta \chi_\eps|\le C \eps^{1-\delta}$ to get the last inequality.

By the Friedrichs lemma, commuting the convolution with the long range Hamiltonian $H^L$, one gets an error term which is small in the $L^2$ norm as $\eps \to 0$.  Since
\begin{multline}  \label{eq:termbyterm}
 \norm{ \Phi_\eps * [ w^{-1} (H^L-z)u] - w^{-1}(H^L-z)v_\eps} \le \\
  \norm{ \Phi_\eps * ( w^{-1} \Delta u) - w^{-1} (\Delta v_\eps)}   + 2\norm{ \Phi_\eps * ( w^{-1} W^L \cdot \nabla  u) - w^{-1} (W^L \cdot \nabla  v_\eps)}  \\ + \norm{ \Phi_\eps * [ w^{-1} (V^L-z)u] - w^{-1}(V^L-z)v_\eps} ,
\end{multline}
we can verify this term by term. First, if  $1/2<\delta< 1$ (which can be assumed without  loss of generality),   $w^{-1}$  and all its derivatives are Lipschitz functions  in $\mR^n$. As a consequence, as $\eps \to 0$, 
\[  \norm{ \Phi_\eps * ( w^{-1} \Delta u) - w^{-1} (\Delta v_\eps)} = o(1),\]
applying  Lemma \ref{lemma_firedrich}. To control the last two terms in the same way, we need  to impose the long range  conditions $w^{-1} |\nabla V^L| \le C$ and $w^{-1} |\nabla W^L| \le C$ on the potentials so that both $w^{-1}V^L$ and $w^{-1} W^L$ have bounded gradients in $\mR^n$. Then, using this in \eqref{eq:termbyterm}  yields
\begin{equation} \label{eq:o_eps2}
 \norm{ \Phi_\eps * [ w^{-1} (H^L-z)u] - w^{-1}(H^L-z)v_\eps}_{H^{-1}} = o(1) .
  \end{equation}
  
As a consequence, using \eqref{eq:o_eps1} and \eqref{eq:o_eps2} and using the fact that one has $w^{-1} (H^L-z)u \in H^{-1}$, we get that
\begin{align*}
     &\norm{ w^{-1}(H^L-z)u_\eps - w^{-1} (H^L-z)u}_{H^{-1}} \\
     & \quad \leq \norm{ w^{-1}(H^L-z)u_\eps - \chi_{\eps} w^{-1} (H^L-z)v_{\eps}}_{H^{-1}} \\
     & \qquad + \norm{ \chi_{\eps} w^{-1}(H^L-z)v_\eps - \chi_{\eps} \Phi_\eps * [ w^{-1} (H^L-z)u]}_{H^{-1}} \\
     & \qquad + \norm{ \chi_{\eps} \Phi_\eps * [ w^{-1} (H^L-z)u]- w^{-1} (H^L-z)u}_{H^{-1}}= o(1),
\end{align*}
and hence $w^{-1} (H^L-z) u_\eps \to  w^{-1} (H^L-z) u$  in $H^{-1}$.  We can use now \eqref{eq:S_bound}  to conclude that  $w^{-1} (H-z) u_\eps \to  w^{-1} (H-z) u$ when $\eps \to 0$. This shows that \eqref{eq:Hest1} holds (with $a=0$) for any $u \in H^1$.

Let us introduce now the resolvent operator $R(z) = (H-z)^{-1}$. Under the conditions assumed on the potentials, $H$ is self-adjoint (see Proposition \ref{prop_self_adjoint} in the Appendix). As a consequence $R(z) $ is well defined for every $z\in \mC$ such that $\Im(z) \neq 0$ and it satisfies the estimate
\[ \norm{R(z) f}  \le \frac{1}{\abs{\mathrm{Im}(z)}} \norm{f}, \]
for every $f\in L^2$. This means that $R(z)g$, $\Im(z) \neq 0$, is well defined for $g\in L^2_\delta \subset L^2$. Also, if $g\in C^\infty_c$ the previous estimates imply that $u = R(z)g\in H^1$ (or $H^2$, if \eqref{shortrange2} holds). To see this, notice we have that $Hu = g + z u$, and since $u \in L^2$, all the terms in $Hu$ must have at least $H^{-1}$ regularity, except perhaps for the  term $\Delta u$ (for the short range perturbation see \eqref{eq:S_bound}). But since $g$ is smooth,  we must also have $\Delta u \in H^{-1}$ which shows that $u \in H^1$.

Therefore we can apply \eqref{eq:Hest1} with $a=0$ to the function $u = R(z)g$, taking $g= h\partial^{\alpha_2}f$, for  $f\in C^\infty_c$ and $|\alpha_2| \le 1$. With this choice of $g$ we can finally get rid of the semiclassical $H_{scl}^{-1}$ norm in the right hand side of \eqref{eq:Hest1}, and using that $h^{-2}= \Re(z) = \lambda$, this yields
\begin{equation*} 
\lambda \norm{   R(z) \partial^{\alpha_2} f}^2_{L^2_{-\delta}} + \norm{  \partial^{\alpha_1} R(z) \partial^{\alpha_2} f}_{L^2_{-\delta}}^2 \le    C(\delta, V,W)  \lambda^{\abs{\alpha_2}} \norm{  f }_{{L^2_{\delta}}}^2 ,
\end{equation*}
for every $f \in C^\infty_c$ and $|\alpha_1|=1$. This estimate is the same as \eqref{eq:main_estimate1}, and since  $ C^\infty_c$ is dense in $ L^2_\delta$ it can be extended for every $ f \in L^2_\delta$. This is enough to finish the proof of \eqref{eq:main_estimate1}. As mentioned previously, the proof of \eqref{eq:main_estimate2} from \eqref{eq:Hest1} with $a=1$ is completely analogous to the case $a=0$. This concludes the proof of the main theorem.
\end{proof}

\section{The limiting  absorption principle } \label{sec:lap}

In this section we prove Theorem \ref{thm:main2} from Theorem \ref{thm:main1}.  The fact that one can define  the  resolvent $R(\lambda \pm i0)$ as a bounded operator between the $L^2_\delta$  and $L^2_{-\delta}$ spaces is known as the limiting absorption principle. 

To define the resolvent when $\Im(z) = 0$ one needs show that the limit $R(\lambda \pm i0) f = \lim_{\eps \to  0} R(\lambda \pm i \eps) f$  exists in $L^2_{-\delta}$. This follows from \eqref{eq:main_estimate2} if one can show that
\begin{equation} \label{eq:SRC}
 \norm{\partial_r u - i\lambda u}_{L^2_{\delta -1}} \le C \norm{(H-z) u }_{L^2_{\delta}},  
 \end{equation}
holds for $1/2<\delta<1$ and $u\in H^1$, or other analogous condition. The previous estimate is known as a Sommerfeld radiation condition, see \cite{RodnianskiTao,TesisMiren} for more details. In our case we do not prove  a Sommerfeld radiation condition  like \eqref{eq:SRC}, we use instead the limiting absorption principle already  proved in \cite[Theorem 30.2.10]{H2}. This holds assuming  that $W^S$ is continuous in addition to the conditions assumed in Theorem \ref{thm:main1}. To state H\"ormander's result we need to introduce the Agmon-H\"ormander space $B$ and its dual $B^*$.
\[ \norm{v}_B =  \sum_{j=1}^\infty  \left( R_j \int_{X_j} |v|^2 \, dx \right )^{1/2}<\infty, \]
\[ \norm{v}_{B^*} =  \sup_{j>0}   \left( R_j^{-1} \int_{X_j} |v|^2 \, dx \right )^{1/2}<\infty ,\]
where $R_0 =0$, $R_j = 2^{j-1}$ for $j>1$ and $X_j = \{x\in \mR^n: R_{j-1}\le |x| \le R_{j}\}$.

\begin{Theorem} \label{thm_hormander} $($\cite[Theorem 30.2.10]{H2}$)$.
Assume that $W$ and $V$ satisfy \eqref{shortlong}--\eqref{shortrange2}. Also, assume that $W^S$ is continuous. Then the eigenvalues $\lambda>0$ of $H$ are of finite multiplicity, and form a set $\Lambda$ which is discrete in $\mR_+$. Moreover, if $\lambda \in \mR_+ \setminus \Lambda$ and $\Re(z) = \lambda$, then $\partial^\alpha R(z) f \to \partial^{\alpha} R(\lambda \pm  i0) f$ in the weak$^*$  topology of $B^*$ for every $f\in B$ and $|\alpha| \le 2$, as $z \to \lambda$ in the respective complex half planes.
\end{Theorem}

With this theorem we can finally define the resolvent operator $R(\lambda \pm  i0) $ in order to prove Theorem \ref{thm:main2}, but we would like to have convergence in the $L_{-\delta}^2$ spaces. This follows from the next brief lemma.
\begin{Lemma} \label{weak_star} 
Let $(u_j)_{j\in \mathbb{N}}$ be a sequence in $B^*$ such that  $ u_j \rightarrow u$ in the weak$^*$ topology of $B^*$. Then $ u_j \rightarrow u$ converges weakly in $L^2_{-\delta}$.
\end{Lemma}
\begin{proof}
It follows directly from the fact that $ \norm{v}_{B} \le C \norm{v}_{L^2_{\delta}}  $ (and hence that $L^2_\delta \subset B$ and $B^* \subset L^2_{-\delta}$ continuously).
\end{proof}
\begin{proof}[Proof of Theorem \ref{thm:main2}]
 By the previous lemma and Theorem \ref{thm_hormander} we have that, for every $f\in L^2_{\delta}$ and $\lambda \in \mR_+ \setminus \Lambda$,  $\partial^\alpha R(z) f \to  \partial^\alpha R(\lambda \pm i0) f$ converges weakly in $L^2_{-\delta}$ as $\Im(z) \to 0$.  Now, let $f\in L^2_\delta$. By Theorem  \ref{thm:main1}, we have that if $\Im(z) \neq 0$, $\partial^\alpha R(z) f$ is bounded in $L^2_{-\delta}$, and the right hand side of \eqref{eq:main_estimate2} is independent of $\Im(z)$. Since bounded sets are precompact in the weak topology, this implies that there is a positive sequence   $(\eps_j)_{j \geq 1}$, $\eps_j \rightarrow 0$,  such that 
 \[
 \partial^\alpha R(\lambda  \pm i\eps_j )f \rightharpoonup \partial^\alpha R(\lambda \pm i0)f\quad  \text{weakly in} \; L^2_{-\delta}.
\]
 As a consequence, $\norm{\partial^\alpha R(\lambda \pm i0)f}_{L^2_{-\delta}} \leq \liminf_{j \to \infty} \, \norm{\partial^\alpha R(\lambda  \pm i\eps_j )f}_{L^2_{-\delta}}$. Theorem  \ref{thm:main1}  yields directly the estimate \eqref{eq:main_thm2}.
\end{proof}

\appendix

\section{} \label{sec_appendix}

We now show that $H$ is self-adjoint with form domain $H^1$. We define the sesquilinear form $q_H(u,v) := (u,Hv)$ for $u \in H^1$ and $v \in  C^\infty_c$. Under these assumptions, by integration by parts one can show that
\begin{equation} \label{eq:self_adjoint}
 q_H(u,v) = (Du, D v) + (Du,Wv) + (Wu,Dv) + (u,V v).
 \end{equation}
Then, since $W \in L^\infty(\mR^n, \mR^n) $,  $q_H(u,v)$ makes sense for all $u,v\in H^1$.
\begin{Proposition} \label{prop_self_adjoint}
Let  $W \in L^\infty(\mR^n, \mR^n) $ and  $V \in L^\infty(\mR^n, \mR) $. Then there is a unique self-adjoint operator $H$ with form domain $H^1$, such that \eqref{eq:self_adjoint} holds for all $u,v \in H^1$.
\end{Proposition}
\begin{proof}
The proof follows from \cite[Theorem X.17]{reedsimon}. Thanks to this theorem, it is enough to show that the form $q_H$ is relatively bounded with respect to the form associated to the negative Laplacian, that is $q_{-\Delta}(u,v) = (Du, D v)$. This is immediate by Young's inequality. If $u\in H^1$, for every  $\eps >0$ one has 
\[ |(Du,Wu) + (Wu,Du) + (u,V u)| \le \eps \norm{\nabla u}^2 + (\eps^{-1}\norm{W}_{L^\infty}^2 +\norm{V}_{L^\infty}) \norm{u}^2 ,\]
 so actually the relative bound is zero.
\end{proof}

We now give the proof  of a couple of auxiliary results used in the paper.

\begin{proof}[Proof of Lemma \ref{lemma_properties_wheight0}]
We have that
\begin{align*}
\tilde \varphi_0(r) &= \br{r}, \qquad \qquad \qquad \quad \ 
\tilde \varphi_0'(r) = r \br{r}^{-1},\\
\psi_0'(r) &= (2\delta-1) r \br{r}^{-2\delta-1}, \quad
\psi_0''(r) = (2\delta-1)  \br{r}^{-2\delta-3}\left(1-2\delta r^2 \right ).
\end{align*}
First, we combine the first and third terms on the right hand side of \eqref{eq:hess1} and show that
\begin{equation} \label{eq:hess_est}
 \left( e^{\tau \psi} (\tilde \phi_0 \psi_0''(r) + 2 \tilde \phi_0'(r) \psi_0'(r)) \partial_r u,\partial_r u \right ) > \alpha ( e^{\tau \psi} \br{r}^{-2\delta} \partial_r u,  \partial_r u), 
 \end{equation}
for $\alpha = (2-2\delta)(2\delta-1)$. This follows from the fact that
\begin{align*}
   \tilde \varphi_0(r) \psi_0''(r)  + 2\tilde \varphi_0'(r) \psi_0'(r)  &= (2\delta-1) \br{r}^{-2\delta-2}(1+(2-2\delta)r^2) \\
   &> (2-2\delta)(2\delta-1)  \br{r}^{-2\delta},   
 \end{align*}
since $0<2-2\delta<1$. 
Then, using \eqref{eq:hess_est} in \eqref{eq:hess1} we obtain that
\begin{multline} \label{eq:hess3}
 (\phi''(x) \nabla u,\nabla u) \geq  \alpha \tau ( e^{\tau \psi}  \br{r}^{-2\delta} \partial_r u,  \partial_r u) \\ 
 + \tau^2 (e^{\tau \psi} \tilde \phi_0 (\psi_0'(r))^2 \partial_r u,  \partial_r u)  + \tau ( e^{\tau \psi} \tilde \phi_0 \frac{\psi_0'}{r} \nabla^{\perp} u,\nabla^{\perp} u)  .
\end{multline}
Therefore, using that $\tilde \phi_0 \frac{\psi_0'}{r} =  (2\delta-1) \br{r}^{-2\delta}$, and that $\alpha< (2\delta-1)$ we get
\begin{multline} \label{eq:hess4}
 (\phi''(x) \nabla u,\nabla u) \geq  \alpha \tau ( e^{\tau \psi}  \br{r}^{-2\delta} \nabla u,  \nabla u) \\ 
 + \tau^2 (e^{\tau \psi} \tilde \phi_0 (\psi_0'(r))^2 \partial_r u,  \partial_r u).
\end{multline}
 This yields \eqref{eq:hess2}. \eqref{eq:wheight_derivatives} follows by direct computation.
\end{proof} 
\begin{proof}[Proof of Lemma \ref{lemma_shift}]
The proof is similar to \cite[Lemma 3.2]{Vodev}. We start by noticing that \eqref{eq:G_weights} is of Carleman type. Indeed, we define $\widetilde \psi := \tau^{-1}k \log w +\frac{1}{2}\psi$, so that we have $ w^k e^{\frac{1}{2}\tau \psi}=e^{\tau\widetilde \psi} $.
By the conditions assumed on $\phi$ and since $\tau\geq 1$, we have that $|\nabla \widetilde \psi|, |\Delta \widetilde \psi| \le C:=C(k, \psi)$, where the latest constant is independent of $h$ and $\tau$ and we assume it to be greater than $1$. By direct computation we get
\begin{align*}
w^k(G_{h,\tau}-i)w^{-k}v &= (e^{\tau\widetilde \psi}h^2\,(P+2W^L\cdot D + D \cdot W^L +V^L)\,e^{-\tau\widetilde \psi}- i)\, v \\
&= (h^2P-i)v + Q_{h,\tau} v,
\end{align*}
where $Q_{h,\tau}$ is a semiclassical first order  operator defined by
\begin{equation}\label{first_order_op}
\begin{aligned}
 Q_{h,\tau} v &= h^2(- \tau^2 | \nabla \widetilde{\psi} |^2 + \tau\,\Delta \widetilde \psi  +2i \tau W^L\cdot \nabla \widetilde\psi - i\nabla\cdot W^L + V^L)v \\
 &\quad+ 2\, h(\tau \nabla \widetilde \psi -i W^L) \cdot h\nabla v.
\end{aligned}
\end{equation}
Using the Fourier transform, one can easily check that
\begin{equation*}\label{fourier_estimate}
\norm{(h^2P\, -i)^{-1}v }_{H^{1+a}_{scl}}\leq  2\norm{ v }_{H^{-1+a}_{scl}}.
\end{equation*}
We also consider the resolvent identity
\[
(h^2 P -i + Q_{h,\tau}  )^{-1}=(h^2 P -i)^{-1} +(h^2 P -i)^{-1} Q_{h,\tau}  (h^2 P -i + Q_{h,\tau} )^{-1},
\]
which allows us to show that 
\begin{equation}\label{sem_ineq}
\begin{aligned}
  \norm{ (h^2 P -i + Q_{h,\tau}  )^{-1}v}_{H_{scl}^{1+a}} &\leq 2\norm{ v }_{H^{-1+a}_{scl}}  + 2\norm{  Q_{h,\tau}  }_{ \mathcal{L}(H_{scl}^{1+a},\, H_{scl}^{-1+a})}  \\
  & \quad  \quad \quad  \quad \quad \quad \quad \times \norm{ (h^2 P -i + Q_{h,\tau}  )^{-1}v}_{H_{scl}^{1+a}}.
\end{aligned}
\end{equation}
Using \eqref{first_order_op} and that $\nabla \cdot W^L$  we obtain
\[
\norm{  Q_{h,\tau}  }_{ \mathcal{L}(H_{scl}^{1+a},\, H_{scl}^{-1+a})}\leq \frac{1}{4},
\]
whenever 
\[
h< \tau^{-1} \min\left\{1/4, (18C^2(1+\norm{ W^L }_{L^\infty} + \norm{ V^L }_{L^\infty} + a\norm{ \nabla \cdot W^L }_{L^\infty}))^{-1}\right\}.
\]
Considering cases $a=0$ and $a=1$ separately, the  estimate above immediately implies the desired result by absorbing the second term on the right into the left hand side of \eqref{sem_ineq}.
\end{proof}

\bibliographystyle{alpha}

\end{document}